\documentclass{amsart}
\usepackage{hyperref}
\newtheorem{theorem}{Theorem}[section]
\newtheorem{proposition}{Proposition}[section]

\theoremstyle{definition}
\newtheorem{definition}[theorem]{Definition}
\newtheorem{example}[theorem]{Example}
\newtheorem{problem}[theorem]{Problem}

\theoremstyle{remark}
\newtheorem{remark}[theorem]{Remark}
\numberwithin{equation}{section}

\newcommand{\cB}{{\mathcal B}}

\newcommand{\cD}{{\mathcal D}}
\newcommand{\cE}{{\mathcal E}}
\newcommand{\cF}{{\mathcal F}}
\newcommand{\cG}{{\mathcal G}}

\newcommand{\cK}{{\mathcal K}}

\newcommand{\cM}{{\mathcal M}}
\newcommand{\cN}{{\mathcal N}}

\newcommand{\cU}{{\mathcal U}}
\newcommand{\cV}{{\mathcal V}}
\newcommand{\cW}{{\mathcal W}}
\newcommand{\cX}{{\mathcal X}}

\newcommand{\bD}{{\mathbb{D}}}
\newcommand{\bC}{{\mathbb{C}}}

\newcommand{\sbm}[1]{\left[\begin{smallmatrix} #1
		\end{smallmatrix}\right]}

\usepackage{color, bm}
\begin{document}

\title{On de Branges--Rovnyak Kernels Admitting a Complete Pick Factor}

\author{Haripada Sau}
\address{Department of Mathematics, Indian Institute of Science Education and Research Pune, Maharashtra 411008}
\email{hsau@iiserpune.ac.in}
\thanks{The author was supported by the Anusandhan National Research Foundation (ANRF), Government of India, under the Mathematical Research Impact-Centric Support (MATRICS) scheme (Grant No. ANRF/ARGM/2025/000573/MTR)}





\begin{abstract}
For a contractive multiplier $\varphi$ in the multiplier algebra $M(k)$ of a kernel $k$, the associated de Branges--Rovnyak kernel is given by $k^\varphi(x,y) = (1-\varphi(x)\overline{\varphi(y)})k(x,y)$. Motivated by recent developments clarifying the structural and geometric features of reproducing kernel Hilbert spaces associated with kernels admitting a complete Pick factor, we investigate the precise conditions for a general de Branges-Rovnyak kernel to admit a complete Pick factor, thereby extending the framework introduced by Ahmed, Das and Panja (\textit{J. Geom. Anal.}, 2025). In this paper, we characterize the existence of a complete Pick factor for $k^\varphi$ across a broad class of base kernels encompassing both complete Pick and non-complete Pick architectures (such as the Szegő kernel on the polydisk). Our first characterization is formulated in terms of operator-valued holomorphic functions satisfying an interpolation condition. We also show that $k^\varphi$ admits a complete Pick factor if and only if $(\widetilde k)^\varphi$ is itself a complete Pick kernel, where  
$\widetilde k$ is an auxiliary kernel constructed from the given data. Notably, our main result is completely new even when specialized to the classical Szeg\"o kernel of the unit disk. As an application of our framework, we obtain a structural insight into a classical theorem of Chu (\textit{J. Funct. Anal.}, 2020) and provide an alternative proof of a recent result by Luo and Zhu (\textit{Canad. J. Math.}, 2024). The results are illustrated by concrete examples.
\end{abstract}

\maketitle

\section{Introduction}

\subsection{Background}
A \textit{kernel} $k$ on a set $X$ is a positive semi-definite function $k:X\times X\to\bC$, denoted by $k\succeq 0$, if for each finite set of points $\{x_j\}$ in $X$, the matrix $(k(x_i,x_j))$ is positive semi-definite (i.e., $(k(x_i,x_j))\succeq 0$). A kernel is said to be \textit{normalized} at a point $x_0\in X$ if $k(x,x_0)=1$ for every $x\in X$. We do not require the kernel function to be holomorphic in the first variable and anti-holomorphic in the second variable; however, when this property holds, we shall explicitly refer to them as \textit{holomorphic kernels}. 

We denote by $H(k)$ the reproducing kernel Hilbert space (RKHS) associated with a kernel $k$, and by $M(k)$ the corresponding multiplier algebra endowed with the operator norm. It is well known that when $k$ is holomorphic, both $H(k)$ and $M(k)$ comprise spaces of holomorphic functions. We refer the reader to the seminal paper by Aronszajn \cite{Aronszajn} and the monograph by Paulsen and Raghupathi \cite{PaulRaghu} for the foundational theory of reproducing kernel Hilbert spaces.

Ever since their inception by Aronszajn \cite{Aronszajn} (see also the pioneering monograph by Bergman \cite{Bergman}), reproducing kernel Hilbert spaces have remained a fertile ground for researchers working at the vibrant interface of holomorphic function theory and operator theory. Beyond the classical finite-dimensional setting, the landscape of infinite-dimensional RKHSs comprises such canonical structures as the Hardy space of the unit disk, the Drury–Arveson space, the Dirichlet space, and the space of Dirichlet series. Recent advances in the field have elegantly unified these disparate spaces under the singular umbrella of complete Pick spaces. 

The original definition of complete Pick kernels was formulated by Agler and McCarthy \cite{AM_JFA} within the framework of the Pick interpolation problem. We also refer the reader to an earlier paper by McCullough \cite{McCullough1992} for work on kernels on the unit disk $\bD$ for which a version of Carath\'eodory interpolation holds. An elegant characterization of complete Pick kernels was established by Agler and McCarthy in \cite[Theorem 3.1]{AM_JFA} (see also related work by Quiggin \cite{Quiggin1, Quiggin2} and McCullough \cite{McCullough}). This characterization asserts that \textit{a non-vanishing, normalized kernel is a complete Pick kernel if and only if $1-1/k\succeq 0$, or equivalently, if there exists a Hilbert space $\mathcal U$ and a function $u:X\to\cU$ such that for every $x,y\in X$,}
\begin{align}\label{CompletePick}
k(x,y)=\frac{1}{1-\langle u(x),u(y)\rangle_\cU}.
\end{align}
In view of this identification, we define complete Pick kernels as follows:

\begin{definition}
A non-vanishing, normalized kernel $k$ is a \textit{complete Pick kernel} if it is of the form \eqref{CompletePick} for some function $u:X\to\cU$.
\end{definition} 

It is well known that a function $\varphi$ belongs to the unit ball of $M(k)$ if and only if 
\begin{align}\label{Intro_kphi}
k^\varphi(x,y):=(1-\varphi(x)\overline{\varphi(y)})k(x,y)
\end{align}
is itself a kernel. While we reserve the notation $k$ for a general kernel, the symbols $\mathfrak{s}$ and $\mathfrak{s}_2$ are dedicated to the Szeg\"o kernel $\mathfrak{s}(z,w):=(1-z\overline{w})^{-1}$ and the Bergman kernel $\mathfrak{s}_2(z,w):=(1-z\overline{w})^{-2}$ of the unit disk $\bD$, respectively. It is a classical result that $M(\mathfrak{s})=M(\mathfrak{s}_2)=H^\infty$, the algebra of bounded analytic functions on $\bD$. For a contractive $\varphi\in H^\infty$, the associated kernels $\mathfrak{s}^\varphi$ and $\mathfrak s_2^\varphi$ are the celebrated \textit{de Branges–Rovnyak kernels} and \textit{sub-Bergman kernels}, which generate the de Branges–Rovnyak and sub-Bergman spaces, respectively. 

Serving as a rich source of Hilbert spaces of analytic functions, these spaces constitute a foundational concept in function theory, invariant subspace theory, operator model theory, and the theory of Pick interpolation. We refer the reader to \cite{dBR1, dBR2, dBR3, Zhu1996, Zhu2003} for the original work on these kernels, and to \cite{BB, BC, FricainMashreghi2016, HKZ_JLMS, LZ_CJM, Sarason, Timotin} for subsequent developments. Inspired by this classical construction, the kernel $k^\varphi$ defined in \eqref{Intro_kphi} for a general base kernel $k$ and a contractive multiplier $\varphi\in M(k)$ is likewise referred to as the de Branges-Rovnyak kernel. 

Given the dual prominence of complete Pick kernels and de Branges-Rovnyak kernels, the following intriguing open problem was posed in \cite{ADP_JGA2025}:

\begin{problem}\label{P:ADP}
Given a kernel $k$ on a set $X$, characterize those contractive multipliers $\varphi\in M(k)$ for which the de Branges-Rovnyak kernel $k^\varphi$ is a complete Pick kernel.
\end{problem}

Although challenging, this problem was successfully settled in \cite[Theorem 1.3]{ADP_JGA2025} for the class of kernels $k$ admitting a representation of the form
\begin{align}\label{Intro:k}
1-\frac{1}{k(x,y)} = \langle g(x), g(y)\rangle_\cG - \langle f(x), f(y)\rangle_\cF,
\end{align}
where $g:X\to\cG$ and $f:X\to\cF$ are functions mapping into auxiliary Hilbert spaces. This constitutes a remarkably large class of kernels, encompassing all complete Pick kernels as well as prominent non-complete Pick kernels, such as the Szeg\"o kernel of the polydisk $\bD^d$. Furthermore, this characterization elegantly recovers a result of Chu \cite{Chu} concerning the Hardy space of the unit disk $H(\mathfrak{s})=H^2$, and a result of Sautel \cite{Sautel} concerning the Drury-Arveson space of the Euclidean ball $\mathbb B_d$ with its associated kernel $k(\bm z,\bm w) = (1-\langle \bm z,\bm w\rangle)^{-1}$ for $\bm z,\bm w \in \mathbb B_d$.

\subsection{The main problem of interest}
Recent developments demonstrate that kernels admitting a complete Pick factor enjoy a rich function- and operator-theoretic structure, even when they fail to be complete Pick themselves. For instance, Clou\^atre, Hartz, and Schillo established in \cite{CHS_PAMS2019} that the associated spaces admit a version of the celebrated Beurling–Lax–Halmos theorem for characterizing invariant subspaces. When specialized to the complete Pick setting, their result recovers a seminal theorem of McCullough and Trent \cite{McT_JFA2000}. Furthermore, Bhattacharyya and Jindal proved in \cite[Theorem 2.13]{BJ_JOT2025} that every kernel possessing a complete Pick factor admits a characteristic function, thereby extending their earlier work \cite[Theorem 3.2]{BJ_Adv2023} from the standard complete Pick framework. We refer the reader to other recent works \cite{AlHMc_Adv, McTs_JLMS, Ts_JFA} exploring the structure of kernels admitting a complete Pick factor. Motivated by this line of inquiry, we consider the following generalization of Problem \ref{P:ADP}.

\begin{problem}\label{P:ADP'}
Given a kernel $k$ on a set $X$, characterize those contractive multipliers $\varphi\in M(k)$ for which the de Branges-Rovnyak kernel $k^\varphi$ admits a complete Pick factor.
\end{problem}

Part of the motivation for this work stems from an intriguing rigidity phenomenon demonstrated by Ahmed, Das, and Panja \cite[Theorem 4.2]{ADP_JGA2025}. In the context of the Hardy space of the polydisk $\mathbb D^d$ for $d\geq 3$, their result asserts that for the Szegő kernel $k_d$ of $\mathbb D^d$, there is no contractive multiplier $\varphi$ for which $k_d^\varphi$ is complete Pick. Interestingly, this rigidity disappears if we merely require that $k_d^\varphi$ admit a complete Pick factor, rather than requiring it to be one. Indeed, considering the multiplier $\varphi:\bD^3\to \bD$ given by $\varphi(z_1,z_2,z_3)=z_1$, the resulting kernel $k_3^\varphi$ is the product of the complete Pick kernels $(1-z_2\overline{w_2})^{-1}$ and $(1-z_3\overline{w_3})^{-1}$. 

A similar rigidity phenomenon was established for the Bergman kernel by Shen, Tian, and Yang \cite{STY-NYJM}, who proved that if $b_n$ denotes the Bergman kernel of $\bD^n$ for $n\geq 2$, there is no non-constant contractive analytic function $\varphi$ on $\bD^n$ such that $b_n^\varphi$ is complete Pick. As with the Szeg\"o kernel, an analogous coordinate-projection example demonstrates that this rigidity vanishes within the broader framework of Problem \ref{P:ADP'}. 

We also note the interesting earlier work of Luo and Zhu \cite{LZ_CJM}, which proves that the only contractive analytic functions $\varphi$ on $\mathbb D$ for which the sub-Bergman kernel $\mathfrak{s}_2^\varphi$ is complete Pick are the M\"obius transformations. In stark contrast, for every contractive analytic function $\varphi$ on $\bD$, the kernel $\mathfrak{s}_2^\varphi$ automatically admits a complete Pick factor—namely, the Szeg\"o kernel $\mathfrak{s}$. Given the rich repository of examples supporting an affirmative answer to Problem \ref{P:ADP'}, it is natural to seek structural representations of the functions for which Problem \ref{P:ADP'} can be solved constructively.

\subsection{Main results}
Our first main result, Theorem \ref{T:Main}, answers Problem \ref{P:ADP'} for kernels of the form \eqref{Intro:k}. This result appears to be new even in the classical setting of the unit disk. Regardless of whether $k$ is holomorphic, the result provides a characterization in terms of an operator-valued holomorphic function on the unit disk $\bD$ satisfying an interpolation condition (see item (2) of Theorem \ref{T:Main}). 

Furthermore, Theorem \ref{T:Main} item (3) shows that solving Problem \ref{P:ADP'} for a kernel $k$ of the form \eqref{Intro:k} is equivalent to solving Problem \ref{P:ADP} for an appropriately constructed auxiliary kernel $\widetilde k$ again of the form \eqref{Intro:k}. More precisely, given a kernel $k$ as in \eqref{Intro:k}, we construct a kernel $\widetilde k$ of the form \eqref{Intro:k} and show that $k^\varphi$ admits a complete Pick factor if and only if $(\widetilde k)^\varphi$ is a complete Pick kernel. The kernel $\widetilde k$, however, depends not only on $k$ but also on the non-complete Pick factor of $k^\varphi$. We will discuss how our result naturally recovers the main result of \cite{ADP_JGA2025} (and consequently, the results of Chu \cite{Chu} for the Szeg\"o kernel on $\bD$ and Sautel \cite{Sautel} for the Drury-Arveson kernel on the Euclidean Ball) when specialized to the complete Pick setting.

The approach adapted in \cite{ADP_JGA2025} heavily relies on two foundational results for complete Pick kernels: Leech's Theorem \cite{Leech} and a theorem due to Ball, Trent, and Vinnikov \cite{BTV}. To the author's knowledge, neither of these tools is available for general kernels admitting a complete Pick factor. The strategy developed in this paper instead makes crucial use of the now-classical lurking isometry argument.

After establishing our first main result in Section 2, we turn to its applications and utilities in Section 3. Applying Theorem \ref{T:Main} to Szeg\"o-type kernels of the form
\begin{align}\label{Intro_class-k}
k_g(x,y)=\frac{1}{1-g(x)\overline{g(y)}}
\end{align}
for some $g:X\to\bD$ yields a more detailed and succinct representation of the functions $\varphi$ for which $k_g^\varphi$ admits a complete Pick factor. We also apply our framework to Bergman kernels to provide an alternative, independent proof of the result by Luo and Zhu \cite{LZ_CJM}. We summarize these main applications below (the first two of which are in the context of Problem \ref{P:ADP}):

\begin{enumerate}
    \item Chu \cite{Chu} proved that a contractive function $\varphi\in H^\infty$ makes $\mathfrak{s}^\varphi$ a complete Pick kernel if and only if there exists a contractive $\psi\in H^\infty$ such that $\varphi(z)=z\psi(z)$. We show that this condition is further equivalent to the existence of an auxiliary Hilbert space $\cX$, vectors $\mathfrak b,\mathfrak{c}\in \cX$, an operator $D$ on $\cX$, and a scalar $\alpha\in\bD$ such that the block operator matrix
    $$
    \begin{bmatrix}
        \alpha &\mathfrak{b}^*\\
        \mathfrak{c}&D
    \end{bmatrix}:
    \begin{bmatrix}
        \bC \\
        \cX
    \end{bmatrix} \to        \begin{bmatrix}
        \bC \\
        \cX
    \end{bmatrix}
    $$
    is a contraction, and the function $\varphi$ can be represented in the form 
    $$
    \varphi(z) = \frac{\alpha z}{1-z\varphi(z)\langle (I-\varphi(z)D)^{-1}(\mathfrak{c}) , \mathfrak{b}\rangle_\cX}.
    $$
    In view of Chu's result, a moment's reflection reveals that for $\mathfrak{s}^\varphi$ to be complete Pick, it is necessary for $\varphi$ to be injective. We shall present an explicit non-injective function $\varphi$ such that $\mathfrak{s}^\varphi$ admits a non-trivial complete Pick factor (where a trivial complete Pick kernel denotes the constant function $1$).

    \item We apply our method to the specific context of the Bergman kernel $\mathfrak{s}_2$ of $\bD$ to provide a new proof of the result of Luo and Zhu \cite{LZ_CJM}, which asserts that \textit{the only contractive functions $\varphi\in H^\infty$ for which $\mathfrak{s}_2^\varphi$ is a complete Pick kernel are the automorphisms of $\bD$.} This assertion was also deduced in \cite{ADP_JGA2025} from their main theorem.

    \item We apply Theorem \ref{T:Main} to the class of kernels $k_g$ as in \eqref{Intro_class-k}, proving that $k_g^\varphi$ admits a complete Pick factor if and only if there exists an analytic function $\Phi:\bD^2\to \bD$ of a specific structure (see Theorem \ref{T:Classical}) satisfying the interpolation condition:
    $$
    \Phi(g(x),\varphi(x)) = \varphi(x)
    $$ 
    for each $x\in X$.
\end{enumerate}

\section{De Branges-Rovnyak kernels admitting a complete Pick factor}
We assume the reader is familiar with the basics of reproducing kernel Hilbert spaces; thorough treatments can be found in \cite{Aronszajn, PaulRaghu}. To set the stage, we first recall the main result of \cite[Theorem 1.3]{ADP_JGA2025}.
\begin{theorem}[Ahmed--Das--Panja \cite{ADP_JGA2025}]\label{T:ADP}
    Let $k$ on a domain $\Omega$ in $\bC^d$ be a non-vanishing kernel for which there exist Hilbert spaces $\cF,\cG$ and functions $f:\Omega\to \cF$, $g:\Omega\to\cG$ such that
    $$
    1-\frac{1}{k(x,y)} = \langle g(x),g(y)\rangle_\cG - \langle f(x),f(y)\rangle_\cF.
    $$ Let $k$ be normalized at a point $w_0\in \Omega$ and $\varphi:\Omega\to\bD$ be in $M(k)$ so that $\varphi(w_0)=0$. Then the following are equivalent:
    \begin{enumerate}
        \item The function $\left( 1-\varphi(x)\overline{\varphi(y)}\right)k(x,y)$ is a complete Pick kernel.
        \item There exists a contractive analytic function 
$        \Psi:\bD\to \cB\left(\cG,
        \sbm{
           \bC\\
    \cF
        }\right)$
        so that
        $$
        \Psi(\varphi(x)) g(x) = \sbm{\varphi(x) \\ f(x)}
        $$for every $x\in\Omega$.
    \end{enumerate}
\end{theorem}
The following terminology is convenient for us.
\begin{definition}
    Let $X$ be a set and $\cU$ be a Hilbert space. A function $u:X\to\cU$ is said to be \textit{non-vanishing} if the associated kernel function
$    (x,y)\mapsto \langle u(x),u(y)\rangle$ is non-vanishing in the usual sense. The function $u$ is called {\em normalized} at a point in $X$, if the associated kernel function is so at the same point.
\end{definition}   
In the theorem below, we do not require the kernel $k$ to be holomorphic (in the first coordinate and anti-holomorphic in the second coordinate). However, if $k$ is holomorphic, then all the ensuing kernels and functions in the theorem are holomorphic as well. The key fact that makes it possible is the Kolmogorov decomposition (see e.g., \cite[Theorem 2.3]{AMc_PickBook}) that asserts that if $k:X\times X\to\bC$ is a holomorphic kernel, then there is a Hilbert space $\cK$ and a holomorphic function $f:X\to\cK$ so that $k(x,y)=\langle f(x), f(y)\rangle_\cK$ for every $x,y\in X$.
\begin{theorem}\label{T:Main}
    Let $X$ be any set and $k$ on $X$ be a non-vanishing kernel such that
    \begin{align}\label{k}
    1-\frac{1}{k(x,y)} = \langle g(x),g(y)\rangle_\cG - \langle f(x),f(y)\rangle_\cF
    \end{align}
    where $\cG$, $\cF$ are Hilbert spaces, and $g:X\to\cB(\cG)$, $f:X\to\cB(\cF)$ are functions. Also suppose that $k$ is normalized at $x_0\in X$. Let $\varphi:X\to\bD$ be so that $\varphi(x_0)=0$. Then the following are equivalent:
    \begin{enumerate}
        \item There exists a non-vanishing, normalized (at $x_0$) kernel $\ell$ on $X$, and a Hilbert space $\cU$-valued function $u$ on $X$ such that 
    \begin{align}\label{CNP_Factor}
            \left(1-\varphi(x)\overline{\varphi(y)}\right)k(x,y) = \frac{\ell(x,y)}{1-\langle u(x),u(y)\rangle_\cU}.
        \end{align}
        \item There exist a Hilbert space $\cV$ and a non-vanishing, normalized (at $x_0$) function $v:X\to \cV$, and a holomorphic function  
        \begin{align}\label{ThePsi}
            \Psi:\bD\to \cB\left(\begin{bmatrix}
                \bC\\
                \cV\otimes\cG
            \end{bmatrix},
            \begin{bmatrix}
                \bC\\
                \cV\otimes \sbm{\bC\\ \cF}
            \end{bmatrix}\right)
        \end{align}so that for every $x\in X$,
        \begin{align}\label{PsiActs}
            \Psi(\varphi(x))\begin{bmatrix}
                1\\
                v(x)\otimes g(x)
            \end{bmatrix}=
            \begin{bmatrix}
                \varphi(x)\\
                v(x)\otimes\sbm{1\\f(x)}
            \end{bmatrix}.
        \end{align}
        \item There exist a Hilbert space $\cW$, a non-vanishing, normalized (at $x_0$) function $w:X\to \cW$, and a kernel $\widetilde k$ on $X$ of the form
        \begin{align}\label{tildeK}
            1-\frac{1}{\widetilde k(x,y)}=
            \langle \widetilde g(x),\widetilde g(y)\rangle_{\widetilde \cG} -\langle \widetilde f(x),\widetilde f(y)\rangle_{\widetilde \cF}
        \end{align}where
        \begin{align}\label{tildef&g}
&\notag        \widetilde g(x)=\begin{bmatrix}
            1\\
            w(x)\otimes g(x)
        \end{bmatrix}\in \begin{bmatrix}
            \bC \\
            \cW\otimes\cG
        \end{bmatrix}=:\widetilde\cG\mbox{ and }\\
        &
        \widetilde f(x)=\begin{bmatrix}
            w(x)\\
            w(x)\otimes f(x)
        \end{bmatrix}\in \begin{bmatrix}
            \cW \\
            \cW\otimes\cF
        \end{bmatrix}=:\widetilde \cF
        \end{align}such that
        $$
        \left( 1-\varphi(x)\overline{\varphi(y)} \right)\widetilde k(x,y)
        $$is a complete Pick kernel.

    \end{enumerate}
\end{theorem}
\begin{proof}We first prove (1)$\Longrightarrow$(2). By applying the Kolmogorov Decomposition to the kernel $\ell$, we obtain a Hilbert space $\cV$-valued function $v$ on $X$ so that $\ell(x,y)=\langle v(x),v(y)\rangle_\cV$ for every $x,y$ in $X$. A rearrangement of the terms in \eqref{CNP_Factor} and the form \eqref{k} of the kernel $k$ gives the following
\begin{align*}
\frac{1-\varphi(x)\overline{\varphi(y)}}
{1-\langle g(x),g(y)\rangle_\cG + \langle f(x),f(y)\rangle_\cF}=
\frac{\langle v(x),v(y)\rangle_\cV}{1-\langle u(x),u(y)\rangle}.
\end{align*} Cross-multiplication followed by a rearrangement yields
\begin{align}\label{KeyEqn}
&    1+\langle v(x)\otimes g(x),v(y)\otimes g(y)\rangle_{\cV\otimes\cG} + \langle \varphi(x) u(x), \varphi(y) u(y)\rangle_\cU \\ \notag
=& \;
    \varphi(x)\overline{\varphi(y)} + \langle v(x),v(y) \rangle_\cV+ \langle v(x)\otimes f(x),v(y)\otimes f(y)\rangle_{\cV\otimes \cF}+\langle u(x),u(y)\rangle_\cU .
\end{align}Let us pause here to note that if the factor $\langle v(x),v(y)\rangle_\cV$ was constant $1$, i.e., the case when the de Branges-Rovnyak kernel $k^\varphi$ is complete Pick, then we shall have a cancellation of the term $1$ on the left hand side with the term $\langle v(x),v(y)\rangle_\cV$ on the right hand side of \eqref{KeyEqn}. This will simplify greatly the argument that follows. We come back to this case in Remark \ref{R:ADP} below.

Define the subspaces
\begin{align}\label{Mspace}
    \cM = \bigvee\left\{
    \begin{bmatrix}
        1\\
        v(x)\otimes g(x)\\
        \varphi(x)u(x)
    \end{bmatrix}:x\in X
    \right\}\subset 
    \begin{bmatrix}
        \bC\\
        \cV\otimes \cG \\
        \cU
    \end{bmatrix}=:\mathfrak{H}
\end{align}and
\begin{align}\label{Nspace}
    \cN = \bigvee\left\{
    \begin{bmatrix}
        \varphi(x)\\
        v(x)\otimes \sbm{1\\f(x)}\\
        u(x)
    \end{bmatrix}:x\in X
    \right\}\subset 
    \begin{bmatrix}
        \bC\\
        \cV\otimes \sbm{\bC\\\cF} \\
        \cU
    \end{bmatrix}=:\mathfrak{K}.
\end{align}The identity \eqref{KeyEqn} implies that the assignment
\begin{align}\label{assignment}
\begin{bmatrix}
        1\\
        v(x)\otimes g(x)\\
        \varphi(x)u(x)
    \end{bmatrix} \mapsto
\begin{bmatrix}
        \varphi(x)\\
        v(x)\otimes \sbm{1\\f(x)}\\
        u(x)
    \end{bmatrix}\quad\mbox{for every }x\in X
\end{align}defines an isometry from $\cM$ onto $\cN$ when extended linearly and continuously. Let $V$ denote a contractive extension of this isometry from $\mathfrak{H}$ into $\mathfrak{K}$, where the Hilbert spaces $\mathfrak{H}$ and $\mathfrak{K}$ are as defined in \eqref{Mspace} and \eqref{Nspace}, respectively. It is convenient to denote this contractive operator by a block operator matrix as follows:
    $$
    V=\begin{bmatrix}
        A&B\\
        C&D
    \end{bmatrix}:\begin{bmatrix}
        \sbm{\bC \\ \cV\otimes\cG} \\ \cU
    \end{bmatrix}\to
    \begin{bmatrix}
        \sbm{\bC \\ \cV\otimes\sbm{\bC\\ \cF}} \\ \cU
    \end{bmatrix}.
    $$The operator $V$ is supposed to satisfy the mapping \eqref{assignment} and so we have
    \begin{align*}
&        A\sbm{1\\ v(x)\otimes g(x)} +\varphi(x) Bu(x) = \sbm{\varphi(x) \\ v(x)\otimes \sbm{1 \\ f(x)}}\quad\mbox{and}\\
&        C\sbm{1\\ v(x)\otimes g(x)} +\varphi(x) Du(x) = u(x)
    \end{align*}for every $x\in X$. The second equality above yields the following:
    $$
    u(x) = (I - \varphi(x)D)^{-1}C \sbm{1\\ v(x)\otimes g(x)}
    $$for every $x\in X$. We put this representation of the function $u$ in the first equation above to obtain
    \begin{align}\label{Realization}
        \sbm{\varphi(x) \\ v(x)\otimes \sbm{1 \\ f(x)}}=A\sbm{1\\ v(x)\otimes g(x)} +\varphi(x) B(I - \varphi(x)D)^{-1}C \sbm{1\\ v(x)\otimes g(x)}.
    \end{align}This shows that if we define the function $\Psi$ for $z\in\bD$ by
    $$
    \Psi(z)=A+zB(I-zD)^{-1}C,
    $$then by the classical Realization Formula, $\Psi$ is a contractive analytic function on $\bD$ taking values in the space of operators from $\sbm{\bC\\ \cV\otimes\cG}$ into $\sbm{\bC \\ \cV\otimes\sbm{\bC\\ \cF}}$. We refer the readers to \cite[Section 6.2]{AMc_PickBook} for the precise statement and a proof of the Realization Formula.  The identity \eqref{Realization} establishes (2).

We next prove (2)$\Longrightarrow$(1). With the function $v$ as in (2), we show that (1) holds for the kernel $\ell(x,y)=\langle v(x),v(y)\rangle_\cV$. Since $v$ is assumed to be non-vanishing, so is the kernel $\ell$. Thus we need to show that the non-vanishing, normalized (at $x_0$) function
    $$
k'(x,y):=\frac{    \left(1-\varphi(x)\overline{\varphi(y)}\right)k(x,y)}{\langle v(x),v(y)\rangle_\cV}
    $$is a complete Pick kernel. We argue that $1-1/k'$ is a kernel taking values in $\bD$. For this we use the form \eqref{k} of $k$ to compute
    \begin{align}\label{aux3}
        \notag &1-\frac{1}{k'(x,y)} \\
       \notag &= 1 - \frac{\langle v(x),v(y)\rangle \left( 1 - \langle g(x),g(y)\rangle+ \langle f(x),f(y)\rangle \right)}{(1-\varphi(x)\overline{\varphi(y)})}\\
        &= \frac{1-\varphi(x)\overline{\varphi(y)} - \langle v(x),v(y)\rangle + \langle v\otimes g(x),v\otimes g(y)\rangle - \langle v\otimes f(x),v\otimes f (y)\rangle}{(1-\varphi(x)\overline{\varphi(y)})}.
    \end{align} We use item (2) to first note that the function from $X\times X$ into $\bC$ defined by
    \begin{align*}
        (x,y) \mapsto \frac{\left\langle \bigg(I-\Psi(\varphi(y))^*\Psi(\varphi(x))\bigg)
         \sbm{1 \\ v\otimes g(x)},  \sbm{1 \\ v\otimes g(y)} \right\rangle}{1-\varphi(x)\overline{\varphi(y)}}
    \end{align*}is a kernel. This is because, by item (2), $\Psi$ is a contractive analytic function on $\bD$ and $\varphi$ takes values in $\bD$. Using \eqref{PsiActs}, we compute
    \begin{align*}
        &\left\langle \bigg(I-\Psi(\varphi(y))^*\Psi(\varphi(x))\bigg)
        \sbm{1 \\ v\otimes g(x)}, 
        \sbm{1 \\ v\otimes g(y)} \right\rangle \\
        &= \left\langle \sbm{1 \\ v\otimes g(x)}, \sbm{1 \\ v\otimes g(y)} \right\rangle
        -
        \left\langle \sbm{\varphi(x) \\ v(x)\otimes \sbm{1\\ f(x)}}, \sbm{\varphi(y) \\ v(y)\otimes \sbm{1\\ f(y)}} \right\rangle\\
        &= 1+ \langle v\otimes g(x),v\otimes g(y) \rangle -\varphi(x)\overline{\varphi(y)}-\langle v(x),v(y) \rangle - \langle v\otimes f(x),v\otimes f(y)\rangle.
    \end{align*} Thus we have shown that $1-1/k'\succeq 0$. To show that this kernel takes values in $\bD$, it is enough to show that $1-1/k'(x,x)$ is less than one. All this together will make $k'$ a complete Pick kernel. We use the expression \eqref{aux3} to compute
    \begin{align*}
        1-\frac{1}{k'(x,x)} = \frac{1-|\varphi(x)|^2 - \|v(x)\|^2 (1 - \|g(x)\|^2 + \|f(x)\|^2)}{1-|\varphi(x)|^2}.
    \end{align*}This gives us the desired conclusion because $1 - \|g(x)\|^2 + \|f(x)\|^2 = 1/k(x,x)>0$. This completes the proof of (2)$\Longrightarrow$(1).
    \medskip

    \noindent
    \textit{Proof of (2)$\Longleftrightarrow$(3):} To obtain item (3) from item (2), let us take $(\cW,w)=(\cV,v)$, where the space $\cV$ and the function $v$ are as in (2). We rearrange the terms in \eqref{tildeK} to get
    \begin{align}\label{TildeAvatar}
        \widetilde k(x,y) = \frac{1}{1-\langle \widetilde g(x),\widetilde g(y) \rangle + \langle \widetilde f(x),\widetilde f(y) \rangle }.
    \end{align} Using the definitions of the functions $\widetilde g$ and $\widetilde f$ as in \eqref{tildef&g} we see that the above expression of $\widetilde k$ is the same as
      \begin{align}\label{TildeAvatar'}
        \widetilde k(x,y) = \frac{1}{\langle v(x),v(y) \rangle \left( 1-\langle  g(x), g(y) \rangle + \langle  f(x), f(y) \rangle\right) } = \frac{k(x,y)}{\langle v(x),v(y) \rangle}.
    \end{align}
    We need to first show that $\widetilde k$ is a kernel. We have already shown that item (2) implies \eqref{CNP_Factor} with $\ell(x,y)=\langle v(x),v(y)\rangle$, i.e., from item (2) we have deduced the following identity of functions
    $$
    (1-\varphi(x)\overline{\varphi(y)})k(x,y)=\frac{\langle v(x),v(y) \rangle}{1-\langle u(x),u(y) \rangle},
    $$which is the same as 
    $$
\frac{k(x,y)}{\langle v(x),v(y) \rangle}=\frac{1}{1-\varphi(x)\overline{\varphi(y)}} \cdot \frac{1}{1-\langle u(x),u(y) \rangle}.
    $$The terms on the right hand side is the product of two kernels and hence is a kernel.
    
    Next note that the kernel $\widetilde k$ satisfies the hypothesis \eqref{k} of Theorem \ref{T:ADP}, and so we can apply this theorem to the kernel $\widetilde k$. By part (2) of Theorem \ref{T:ADP}, identity \eqref{PsiActs} is precisely what is required for the kernel $\left( 1-\varphi(x)\overline{\varphi(y)}\right)\widetilde k(x,y)$ to be complete Pick. Theorem \ref{T:ADP} also yields the converse direction (3)$\Longrightarrow$(2) when applied to the kernel $\widetilde k$.
\end{proof}

\begin{remark}\label{R:Normalization}
    We remark here that the normalization assumption in Theorem \ref{T:Main} does not put a restriction on either the kernel $k$ or the function $\varphi$. Indeed, by \cite[Theorem 3.1]{AM_JFA} a non-vanishing kernel $k$ is complete Pick if and only if it is of the form
    $$
    k(x,y) = \frac{\delta(x)\overline{\delta(y)}}{1-\langle u(x), u(y)\rangle_\cU}
    $$for a non-vanishing function $\delta:X\to\bC$ and a Hilbert space valued function $u:X\to\cU$. A kernel of the above form is complete Pick if and only if the kernel $
    (1-\langle u(x), u(y)\rangle)^{-1}
    $ is so. The normalization condition on $k$ removes the rank one factor $\delta(x)\overline{\delta(y)}$. It is convenient to not have to deal with the additional rank one factor. For further elaboration on this, we refer the reader to \cite[Section 3]{AM_JFA}. On the other hand, it is well known that if $k$ is a kernel on $X$, $\varphi$ is contractive in $M(k)$ and $x_0\in X$ is arbitrary, then $k^\varphi$ is complete Pick if and only if $k^{\widetilde\varphi}$ is so where $\widetilde\varphi = m_{\varphi(x_0)}\circ \varphi$ with $m_{\varphi(x_0)}$ denoting an automorphism of $\bD$ that vanishes at $\varphi(x_0)$.
\end{remark}

\begin{remark}\label{R:ADP}
Here we discuss how Theorem \ref{T:ADP}, the main result of Ahmed-Das-Panja \cite{ADP_JGA2025}, can be deduced as a special case of our main result Theorem \ref{T:Main}. This corresponds to the case when $\ell(x,y)=1$ in \eqref{CNP_Factor}.

 Suppose that we have a contractive analytic function $\Psi:\bD\to\cB(\cG, \sbm{\bC \\ \cF})$ as in item (2) of Theorem \ref{T:ADP}. Let us denote this function as a $2\times 2$ block matrix as follows:
\begin{align*}
    \Psi(z) = \begin{bmatrix}
        \Psi_1(z) \\ \Psi_2(z)
    \end{bmatrix}:\cG 
\to \begin{bmatrix}
    \bC \\
    \cF
\end{bmatrix}.
\end{align*}Consider the function $\widetilde\Psi:\bD\to\cB(\sbm{\bC \\ \cG}, \sbm{\bC \\ \bC \\ \cF})$ as follows:
\begin{align*}
    \widetilde\Psi(z)=
    \begin{bmatrix}
        0  & \Psi_1(z) \\
        1 & 0 \\
        0 & \Psi_2(z)
    \end{bmatrix}.
\end{align*}Then it is an easy check that $\widetilde\Psi$ is a contractive analytic function and it satisfies
$$
\widetilde\Psi(\varphi(x)) \begin{bmatrix}
    1 \\
    g(x)
\end{bmatrix} =
\begin{bmatrix}
    \varphi(x) \\
    1\\
    f(x)
\end{bmatrix}
$$This is precisely item (2) of Theorem \ref{T:Main} with $(\cV,v)=(\bC,1)$. By the implication (2)$\Longrightarrow$(1) of Theorem \ref{T:Main}, $(1-\varphi(x)\overline{\varphi(y)})k(x,y)$ is a complete Pick kernel.

We now demonstrate how to obtain (1)$\Longrightarrow$(2) of Theorem \ref{T:ADP} from that of Theorem \ref{T:Main}. Thus we have item (1) of Theorem \ref{T:Main} but with the constant factor $\ell(x,y)=1$. Then the function $v$ in item (2) of Theorem \ref{T:Main} can be chose to be the constant function taking value $1$ in $\bC=\cV$. In this case, as noted in the proof of Theorem \ref{T:Main}, the identity \eqref{KeyEqn} obtained in the proof of (1)$\Longrightarrow$(2) will be simplified. Indeed, all the tensor products will disappear and the term $1$ on the left hand side will cancel out with the term $\langle v(x), v(y)\rangle$ on the right hand side. Tracing the proof of (1)$\Longrightarrow$(2) in Theorem \ref{T:Main} with this cancellation in mind, we arrive at a function  \begin{align}\label{ThePsi'}
            \Psi:\bD\to \cB\left(
 \cG
 ,
            \begin{bmatrix}
 \bC\\ \cF
            \end{bmatrix}\right)
        \end{align}so that for every $x\in X$,
        \begin{align}\label{PsiActs}
            \Psi(\varphi(x))
                g(x)
            =
            \begin{bmatrix}
                \varphi(x)\\
               f(x)
            \end{bmatrix}.
        \end{align}This is precisely the item (2) of Theorem \ref{T:ADP}.
\end{remark}

\section{Some illustrative examples}
To illustrate the utility of our main result (Theorem \ref{T:Main}), we now present some concrete examples of kernels. Here we provide a somewhat new description of $\varphi\in H^\infty$ making the classical de Branges-Rovnyak kernel $\mathfrak{s}^\varphi$ a complete Pick kernel. We also give a new proof of a result due to Luo and Zhu \cite{LZ_CJM} and \cite{ADP_JGA2025}.

\subsection{The Szeg\"o kernel example} Here we illustrate the special case when the kernel $k$ is of the form
\begin{align}\label{classical_k}
k(x,y)=\frac{1}{1-g(x)\overline{g(y)}}
\end{align}for some function $g:X\to\bD$. The case $(X,g)=(\bD,z)$ includes the Szeg\"o kernel for the unit disk. With the notations as in Theorem \ref{T:Main}, we have $\cG=\bC$ and $\cF=\{0\}$. In this case, we have a more succinct necessary and sufficient condition.
\begin{theorem}\label{T:Classical}
    Let $k:X\times X\to\bC$ be a kernel of the form \eqref{classical_k} for some function $g:X\to\bD$ vanishing at a point $x_0\in X$. For a function $\varphi:X\to\bD$ vanishing at $x_0$, the following are equivalent:
    \begin{enumerate}
        \item There exist Hilbert space $\cU$-valued function $u$ on $X$ such that $\|u(x)\|_\cU <1$ and a non-vanishing Hilbert space $\cV$-valued function $v$ on $X$, such that for all $x,y$ in $X$,
        \begin{align}\label{Szego_phi}
           \frac{(1-\varphi(x)\overline{\varphi(y)})}{1-g(x)\overline{g(y)}} = \frac{\langle v(x),v(y)\rangle_\cV}{1-\langle u(x),u(y)\rangle_\cU}.
        \end{align}
        \item There exist an auxiliary Hilbert space $\cX$, and a contractive analytic function $\mathbb U$ on $\bD$:
        $$
        \mathbb U(z):=\begin{bmatrix}
        A(z) & B(z) \\
        C(z) & D(z)\
        \end{bmatrix}:
        \begin{bmatrix}
            \bC \\
            \cX
        \end{bmatrix}\to
        \begin{bmatrix}
            \bC \\
            \cX
        \end{bmatrix}
        $$such that the $\cX$-valued function defined on $X$ by
        \begin{align}\label{v(z)}
X\ni         x\mapsto (I-g(x)D(\varphi(x)))^{-1}C(\varphi(x)) \in \cX
        \end{align}is non-vanishing, and
        with the scalar-valued contractive analytic function defined on the bidisk $\bD^2$ by
        $$
        \Phi(z_1,z_2)= A(z_2) + z_1B(z_2)(I - z_1 D(z_2))^{-1}C(z_2),
        $$ we have
\begin{align}\label{Interesting}
        \varphi(x) = \Phi(g(x),\varphi(x))
        \end{align} for every $x\in X$.
\item There exist auxiliary Hilbert spaces $\cX_1,\cX_2$ and a unitary operator
    $$
    \begin{bmatrix}
        A&B\\
        C&D
    \end{bmatrix}:
    \begin{bmatrix}
        \bC \\
        \sbm{\cX_1 \\ \cX_2}
    \end{bmatrix} \to     \begin{bmatrix}
        \bC \\
        \sbm{\cX_1 \\ \cX_2}
    \end{bmatrix}
    $$so that with the functions $\widetilde v:X\to\cX_1$ and $\widetilde u:X\to\cX_2$ defined by
    \begin{align}\label{TheAuxs}
    \begin{bmatrix}
        \widetilde v(x) \\
        \widetilde u(x)
    \end{bmatrix}=
    \left ( I_{\sbm{ \cX_1 \\ \cX_2}} - D \begin{bmatrix}
        g(x)I_{\cX_1} & 0\\
        0 & \varphi(x)I_{\cX_2}
    \end{bmatrix}\right)^{-1}C
    \end{align}we have $\widetilde v$ non-vanishing and $\|\widetilde u(\cdot)\|<1$ and with
    the contractive analytic function $\Psi:\bD^2\to \bC$ defined by
    \begin{align}\label{BidiskPsi}
    \Psi(z_1,z_2) = A + B \begin{bmatrix}
        z_1I_{\cX_1} & 0\\
        0 & z_2I_{\cX_2}
    \end{bmatrix}\left ( I_{\sbm{ \cX_1 \\ \cX_2}} - D \begin{bmatrix}
        z_1I_{\cX_1} & 0\\
        0 & z_2I_{\cX_2}
    \end{bmatrix}\right)^{-1}C
    \end{align}we have
    \begin{align}\label{Interesting'}
   \varphi(x)= \Psi(g(x),\varphi(x)) .
    \end{align}
    \end{enumerate}

\end{theorem}
\begin{proof}
To show (1)$\Longleftrightarrow$(2),  our strategy is to show that item (2) in this theorem is equivalent to  item (2) of Theorem \ref{T:Main} when applied to the special case \eqref{classical_k}. Item (2) of Theorem \ref{T:Main} in this case is as follows: there exist a Hilbert space $\cV$, a function $v:X\to\cV$ and a contractive holomorphic function 
    $\Psi:\bD\to\cB\left(\sbm{\bC \\ \cV}\right)$ so that
    \begin{align}\label{Goal}
    \Psi(\varphi(x))\begin{bmatrix}
        1\\
        g(x )v(x)
    \end{bmatrix} = \begin{bmatrix}
        \varphi(x) \\
        v(x)
    \end{bmatrix}
    \end{align}for every $x\in X$.

    Let us suppose that item (2) of the present theorem holds. To obtain a contractive holomorphic function 
    $\Psi:\bD\to\cB\left(\sbm{\bC \\ \cV}\right)$ satisfying \eqref{Goal}, choose $\cV=\cX$ and define $v:X\to\cV$ by
    $$
    v(x):= (I - g(x)D(\varphi(x)))^{-1}C(\varphi(x))
    $$for every $x\in X$. This is the same as having
    \begin{align}\label{1half}
        v(x) = g(x)D(\varphi(x))v(x) + C(\varphi(x))
    \end{align}for every $x\in X$. Next we unfold \eqref{Interesting} and use the definition of $v$ to see that for every $x\in X$
    \begin{align}\label{2half}
        \varphi(x)= A(\varphi(x)) + g(x)B(\varphi(x))v(x).
    \end{align}We see that with $\Psi=\mathbb U$, \eqref{Goal} can be read off from \eqref{1half} and \eqref{2half}. Thus item (2) of Theorem \ref{T:Main} is obtained. To obtain item (2) of Theorem \ref{T:Classical} from item (2) of Theorem \ref{T:Main}, the argument given above can be reversed.

    We now prove (1)$\Longrightarrow$(3). We begin by rearranging the terms in \eqref{Szego_phi} to obtain
\begin{align}\label{Lurk}
\notag    1 + \langle \varphi(x) u(x), \varphi(y) u(y)\rangle &+ \langle g(x) v(x), g(y) v(y)\rangle \\
    &= \varphi(x)\overline{\varphi(y)}  + \langle v(x), v(y)\rangle + \langle u(x), u(y)\rangle,
\end{align} which is the same as
    \begin{align*}
   & 1 +
   \left \langle
    \begin{bmatrix}
        g(x)I_{\cV} & 0\\
        0 & \varphi(x)I_{\cU}
\end{bmatrix}\begin{bmatrix}
        v(x) \\
        u(x)
    \end{bmatrix},
    \begin{bmatrix}
        g(y)I_{\cV} & 0\\
        0 & \varphi(y)I_{\cU}
        \end{bmatrix}\begin{bmatrix}
        v(y) \\
        u(y)
    \end{bmatrix}
   \right \rangle \\
   &=\varphi(x)\overline{\varphi(y)} +  \left\langle
   \begin{bmatrix}
        v(x) \\
        u(x)
    \end{bmatrix},\begin{bmatrix}
        v(y) \\
        u(y)
    \end{bmatrix}
    \right\rangle.
    \end{align*}
   Thus, if we denote the spaces
   \begin{align}\label{Mspace'}
       \cM=
       \left\{
\begin{bmatrix}
    1\\
        g(x)v(x)\\
        \varphi(x)u(x)
\end{bmatrix}
:x\in X
       \right\}\subseteq \begin{bmatrix}
           \bC\\
           \cV\\
           \cU
       \end{bmatrix}
   \end{align}
   and
   \begin{align}\label{Nspace'}
       \cN=
       \left\{
\begin{bmatrix}
    \varphi(x)\\
        v(x)\\
        u(x)
\end{bmatrix}:x\in X
       \right\}\subseteq \begin{bmatrix}
           \bC\\
           \cV\\
           \cU
       \end{bmatrix},
   \end{align}then \eqref{Lurk} shows that the assignment 
   \begin{align}\label{Map}
       \begin{bmatrix}
    1\\
        g(x)v(x)\\
        \varphi(x)u(x)
\end{bmatrix} \to \begin{bmatrix}
    \varphi(x)\\
        v(x)\\
        u(x)
\end{bmatrix}\quad\mbox{for every $x\in X$}
   \end{align} defines an isometry from $\cM$ onto $\cN$ when extended linearly and continuously. If required, we can add a Hilbert space of appropriate dimension to $\cU$ (still denoted by $\cU$) to extend the above isometry to a unitary operator 
   $$
   \begin{bmatrix}
       A&B\\
       C&D
   \end{bmatrix}:
   \begin{bmatrix}
       \bC\\
       \sbm{
       \cV\\
       \cU
       }
   \end{bmatrix}\to
   \begin{bmatrix}
       \bC\\
       \sbm{
       \cV\\
       \cU
       }
   \end{bmatrix}.
   $$ Choose $\cX_1=\cV$ and $\cX_2=\cU$ and define the function $\Psi:\bD^2\to\bC$ as in \eqref{BidiskPsi}. By the well-known Realization Formula for the bidisk (see e.g., \cite[Theorem 11.13]{AMc_PickBook}), $\Psi$ is a contractive analytic function. Since the operator $\sbm{A&B\\C&D}$ satisfies the mapping \eqref{Map}, we have
   \begin{align}\label{Lurk1}
       A + B \begin{bmatrix}
           g(x)I_\cV & 0\\
           0 & \varphi(x)I_\cU    \end{bmatrix}\begin{bmatrix}
           v(x) \\ u(x)
       \end{bmatrix} &= \varphi(x) \quad\mbox{and} \\ \notag
       C + D\begin{bmatrix}
           g(x)I_\cV & 0\\
           0 & \varphi(x)I_\cU       \end{bmatrix}\begin{bmatrix}
           v(x) \\ u(x)
       \end{bmatrix} &= \begin{bmatrix}
           v(x) \\
           u(x)
       \end{bmatrix}.
   \end{align}The second equation gives
   \begin{align}
       \begin{bmatrix}
           v(x) \\
           u(x)
       \end{bmatrix} = \left (I- D\begin{bmatrix}
           g(x)I_\cV & 0\\
           0 & \varphi(x)I_\cU
       \end{bmatrix}\right)^{-1}C.
   \end{align}Plugging this in \eqref{Lurk1} we get
   $$
   \varphi(x) = A + B \begin{bmatrix}
           g(x)I_\cV & 0\\
           0 & \varphi(x)I_\cU
       \end{bmatrix} \left (I- D\begin{bmatrix}
           g(x)I_\cU & 0\\
           0 & \varphi(x)I_\cU
       \end{bmatrix}\right)^{-1}C = \Psi(g(x),\varphi(x)).
   $$Thus we have proved item (3).

   Conversely, to obtain item (1) from item (3), we assume that there are spaces $\cX_1,\cX_2$, functions $\widetilde v$, $\widetilde u$ and $\Psi$ as in item (3). The assumption on the functions $\widetilde v$ and $\widetilde u$ ensures that the kernel
   $$
   (x,y) \mapsto \langle \widetilde v(x), \widetilde v(y) \rangle 
   $$is non-vanishing, and 
   $$
   (x,y) \mapsto \frac{1}{1-\langle \widetilde u(x), \widetilde u(y)\rangle}
   $$is a complete Pick kernel on $X$. We shall arrive at \eqref{Szego_phi} with $v=\widetilde v$ and $u=\widetilde u$. We first rewrite \eqref{TheAuxs} to get
   $$
   \begin{bmatrix}
       \widetilde v(x) \\
       \widetilde u(x)
   \end{bmatrix} = C + D\begin{bmatrix}
       g(x)I_{\cX_1} & 0\\
       0 & \varphi(x)I_{\cX_2}
   \end{bmatrix}\begin{bmatrix}
       \widetilde v(x)\\
       \widetilde u(x)
   \end{bmatrix}.
   $$In view of the definition of the function $\Psi$, we rewrite \eqref{Interesting'} to get
   $$
   A+B\begin{bmatrix}
       g(x)I_{\cX_1} & 0\\
       0 & \varphi(x)I_{\cX_2}
   \end{bmatrix}\begin{bmatrix}
       \widetilde v(x)\\
       \widetilde u(x)
   \end{bmatrix} = \varphi(x).
   $$These two identities ensure that the unitary operator $\sbm{A&B\\ C&D}$ satisfies the mapping \eqref{Map} with $(v,u)=(\widetilde v,\widetilde u)$. This will give the equation \eqref{Lurk} for every $x,y\in X$ with $(v,u)=(\widetilde v,\widetilde u)$. Finally equation \eqref{Lurk} is the same as having \eqref{Szego_phi} for $(v,u)=(\widetilde v,\widetilde u)$. This completes the proof of (3)$\Longrightarrow$(1).
\end{proof}

\begin{example}
    We now illustrate Theorem \ref{T:Main} and Theorem \ref{T:Classical} by the example of the Szeg\"o kernel $\mathfrak s$ on $\bD$ and for the function $\varphi(z)=c^2z^2$ for any non-zero $c\in \bD$. Note that 
    $$
    \mathfrak{s}^\varphi(z,w)= \frac{1-c^4z^2\overline{w}^2}{1-z\overline{w}}
    $$cannot be a complete Pick kernel because $\varphi$ is not injective, which is a necessary condition $\varphi$ needs to satisfy for $\mathfrak{s}^\varphi$ to be complete Pick (observed in Subsection 1.3). However, the kernel $\mathfrak{s}^\varphi$ has a complete Pick factor. Indeed,
$$
\mathfrak{s}^\varphi(z,w)= \frac{1-c^4z^2\overline{w}^2}{1-z\overline{w}} = (1+c^2z\overline{w}) \frac{1-c^2z\overline{w}}{1-z\overline{w}}.
$$It is easy to see that the function
$$
1+c^2z\overline{w} = \left\langle \begin{bmatrix}1\\ cz\end{bmatrix}, \begin{bmatrix}1\\ cw\end{bmatrix} \right\rangle_{\bC^2}
$$is a kernel and 
    $$
    \frac{1-c^2z\overline{w}}{1-z\overline{w}}
    $$is a complete Pick kernel because
    $$
    1-\frac{1-z\overline{w}}{1-c^2z\overline{w}} = (1-c^2)z\overline{w}\frac{1}{1-c^2z\overline{w}}\succeq 0.
    $$By Theorem \ref{T:Main} part (2), there exists a contractive analytic function $\Phi:\bD\to \cB(\bC^3,\bC^3)$ so that 
    $$
    \Phi(c^2z^2)\begin{bmatrix}
        1 \\ z \\ cz^2
    \end{bmatrix} = \begin{bmatrix}
        c^2z^2 \\ 1 \\ cz
    \end{bmatrix}
    $$for all $z\in \bD$ and non-zero $c\in\bD$. It is easy to see that this is satisfied by the constant function 
    \begin{align}\label{constant}
    \Phi(w)=\begin{bmatrix}
        0 &0 & c \\
        1 & 0 &0\\
        0 & c& 0 
    \end{bmatrix}
    \end{align}for all $w\in\bD$.

    To see item (2) of Theorem \ref{T:Classical}, take $\cX=\bC^2$ and $\mathbb U$ to be the constant matrix as in \eqref{constant}. It is easy computation to see that the function \eqref{v(z)} is equal to the non-vanishing function $v(z)=\sbm{1\\cz}$ and the holomorphic function $\Phi$ on $\bD^2$ is given by $\Phi(z_1,z_2)=c^2z_1^2$ giving $\Phi(z,c^2z^2)=c^2z^2=\varphi(z)$.
\end{example}

We now turn to the special case of Theorem \ref{T:Classical} where $g$ is holomorphic and $(\cV,v) = (\bC,1)$, i.e., to the case when the de Branges Rovnyak kernel $k^\varphi$ is a complete Pick kernel. This result is new even in the case of the Szeg\"o kernel of $\bD$.

\begin{theorem}\label{T:deBR_CNP}
    Let $k$ be a kernel of the form \eqref{classical_k} for a holomorphic function $g:X\to\bD$, and $\varphi:X\to\bD$ be a holomorphic function such that $\varphi(x_0)=0=g(x_0)$ for some $x_0\in X$. Then the following are equivalent:
    \begin{enumerate}
        \item $k^\varphi$ is a complete Pick kernel.
        \item There exist an auxiliary Hilbert space $\cX$, vectors $\mathfrak a,\mathfrak{b}\in \cX$, an operator $D$ on $\cX$, and $\alpha\in\bD$ so that 
        $$
        \begin{bmatrix}
            \alpha &\mathfrak{b}^*\\
            \mathfrak{c}&D
        \end{bmatrix}:
        \begin{bmatrix}
            \bC \\
            \cX
        \end{bmatrix} \to        \begin{bmatrix}
            \bC \\
            \cX
        \end{bmatrix}
        $$is contractive and
        \begin{align}\label{classical_phi}
        \varphi(x) = \frac{\alpha g(x)}{1-\varphi(x)\psi(x)}
        \end{align}where $\psi:X\to\bD$ is the analytic function given by
\begin{align}\label{classical_psi}
        \psi(x)=g(x)\langle (I-\varphi(x)D)^{-1}(\mathfrak{c}) , \mathfrak{b}\rangle.
        \end{align}
    \end{enumerate}
\end{theorem}
\begin{proof}
    We first show the easy direction (2)$\Longrightarrow$(1): We rearrange the terms in \eqref{classical_phi} and use the form \eqref{classical_psi} of $\psi$ to get 
    \begin{align}\label{reached}
        \varphi(x) = \alpha g(x) + \varphi(x)g(x)(I-\varphi(x)D)^{-1}(\mathfrak{c}) = g(x)\widetilde \psi(\varphi(x))
    \end{align}where $\widetilde\psi:\bD\to\bD$ is the contractive analytic function given by
    \begin{align*}
    \widetilde\psi(z) = \alpha + \langle (I-zD)^{-1}(\mathfrak{c}),\mathfrak{b} \rangle_\cX.
    \end{align*}Apply Remark \ref{R:ADP} to this special case (i.e., $f=0$) to see that  \eqref{reached} implies $k^\varphi$ is a complete Pick kernel.

    For the reverse implication, an application of Remark \ref{R:ADP} does not seem to reveal that the function $\psi$ as defined in \eqref{classical_psi} takes values in $\bD$. We take the following detailed route instead. Suppose $u:X\to\cX$ be a holomorphic function so that $\|u(x)\|<1$ for every $x\in X$ and 
    $$
    \frac{1-\varphi(x)\overline{\varphi(y)}}{1-g(x)\overline{g(y)}}=
    \frac{1}{1-\langle u(x),u(y)\rangle_\cX}.
    $$A rearrangement of the terms gives
    $$
    g(x)\overline{g(y)} + \langle \varphi(x)u(x),\varphi(y)u(y)\rangle_\cX= \varphi(x)\overline{\varphi(y)} + \langle u(x) , u(y)\rangle ,
    $$which leads to the existence of a contraction operator
    $$
   \begin{bmatrix}
            \alpha &\mathfrak{b}^*\\
            \mathfrak{c}&D
        \end{bmatrix}:
    \begin{bmatrix}
        \bC \\ \cX
    \end{bmatrix} \to     \begin{bmatrix}
        \bC \\ \cX
    \end{bmatrix}
    $$where $\mathfrak{b},\mathfrak{c}$ are vectors in $\cX$ so that for every $x\in X$,
    \begin{align}\label{classical_lurk}
    \begin{bmatrix}
        \varphi(x) \\
         u(x)
    \end{bmatrix}=\begin{bmatrix}
        \alpha &\mathfrak{b}^*\\
        \mathfrak c &D
    \end{bmatrix}\begin{bmatrix}
        g(x) \\
        \varphi(x) u(x)
    \end{bmatrix}= \begin{bmatrix}
        \alpha g(x) + \langle \varphi(x)u(x),\mathfrak{b}\rangle \\
        g(x)\mathfrak{c} + \varphi(x)Du(x)
    \end{bmatrix}
    .
    \end{align}The identity obtained by comparing the (21)-entries of \eqref{classical_lurk} shows that
    $$
    u(x) = g(x)(I-\varphi(x)D)^{-1}(\mathfrak{c})
    $$and hence by the Cauchy-Schwarz inequality, the function $\psi(x) = \langle u(x),\mathfrak{b}\rangle $ as defined in \eqref{classical_psi} must take values in $\bD$. Finally, putting the expression of the functions $u(x)$ and $\psi(x)$ in the identity obtained by comparing the (11)-entries of the matrices in \eqref{classical_lurk} gives
    $$
    \varphi(x) = \alpha g(x) + \varphi(x)\psi(x)
    $$giving \eqref{classical_phi} as needed.
\end{proof}


\subsection{The Bergman kernel example}
We shall continue to use the notation $\mathfrak{s}_2$ for the Bergman kernel for the unit disk $\bD$.
As an application of our lurking isometry method, here we give a new proof of the following result due to Luo and Zhu \cite{LZ_CJM}. This result was also deduced independently in \cite{ADP_JGA2025} from the main result therein.

\begin{theorem}
    Let $\varphi:\bD\to\bD$ be an analytic function. Then $\mathfrak{s}_2^\varphi$ is a complete Pick kernel if and only if $\varphi(z)$ is an automorphism of $\bD$.
\end{theorem}
\begin{proof}
    The 'if' direction follows when we note that 
    $    1-1/\mathfrak{s}_2^{cz} = z\overline{w}\succeq0.$
We next prove the 'only if' assertion in two cases.

\textbf{Case I:} Suppose $\varphi(0)=0$. Since $\mathfrak{s}_2^\varphi$ is a non-vanishing, normalized (at $0$) kernel, there is a Hilbert space $\cU$ and an analytic function $u:\bD\to\cU$ so that $\|u(z)\|<1$ for every $z\in \bD$ and
    $$
    \frac{1-\varphi(z)\overline{\varphi(w)}}{(1-z\overline{w})^2} = \frac{1}{1-\langle u(z),u(w)\rangle_\cU}.
    $$As done before, we rearrange the terms to get
    $$
    2z\overline{w} + \langle \varphi(z)u(z), \varphi(w)u(w)\rangle = \varphi(z)\overline{\varphi(w)} + z^2\overline{w}^2 + \langle u(z),u(w)\rangle.
    $$This shows that there exists a contractive block operator matrix 
    $$
    \begin{bmatrix}
        A&B\\
        C&D
    \end{bmatrix}:\begin{bmatrix}
        \bC \\
        \cU
    \end{bmatrix}\to\begin{bmatrix}
        \bC^2 \\
        \cU
    \end{bmatrix}
    $$so that for every $z\in \bD$,
    $$
    \begin{bmatrix}
        A&B\\
        C&D
    \end{bmatrix} \begin{bmatrix}
        \sqrt{2}z\\
        \varphi(z)u(z)
    \end{bmatrix}=\begin{bmatrix}
        \sbm{\varphi(z) \\ z^2} \\ u(z)
    \end{bmatrix}.
    $$This gives the following identities:
    \begin{align}\label{aux4}
    \notag &u(z) = \sqrt{2}z(I-\varphi(z)D)^{-1}C \quad\mbox{and} \\
        &\begin{bmatrix}
            \varphi(z)\\
            z^2
        \end{bmatrix} = \sqrt{2}z (A+ \varphi(z)B(I-\varphi(z)D)^{-1}C) 
    \end{align}To interpret these identities  better, it is convenient to write 
    $$
    A= \begin{bmatrix}
        \alpha_1 \\ \alpha_2
    \end{bmatrix}\in \begin{bmatrix}
        \bC \\ \bC
    \end{bmatrix},\quad 
    B= \begin{bmatrix}
        \mathfrak b_1^* \\ \mathfrak b_2^*
    \end{bmatrix}, \quad C =
        \mathfrak{c}\in \cU
    $$where $\mathfrak{b}_1, \mathfrak{b}_2$ are vectors in $\cU$. With these notations in place, we rewrite \eqref{aux4} as
    \begin{align}\label{aux2}
    \begin{bmatrix}
            \varphi(z)\\
            z^2
        \end{bmatrix} = \sqrt{2}z \begin{bmatrix}
            \alpha_1 + \varphi(z) \langle (I-\varphi(z)D)^{-1}(\mathfrak{c}),\mathfrak{b}_1\rangle \\
            \alpha_2 + \varphi(z) \langle (I-\varphi(z)D)^{-1}(\mathfrak{c}),\mathfrak{b}_2\rangle 
        \end{bmatrix}.
    \end{align}Comparing the (21)-entries we see that $\alpha_2=0$ and since $\varphi(0)=0$, 
    \begin{align}\label{aux}
    \varphi(z) \langle \mathfrak{c},\mathfrak{b}_2\rangle = \frac{z}{\sqrt{2}}.
    \end{align}This shows that $\langle \mathfrak{c},\mathfrak{b}_2\rangle \neq 0$. Now let us put the expression of $\varphi(z)$ from the (11)-entry of \eqref{aux2} in \eqref{aux} to get
    $$
    \sqrt{2}z \langle \mathfrak{c},\mathfrak{b}_2\rangle (\alpha_1 + \varphi(z) \langle (I-\varphi(z)D)^{-1}(\mathfrak{c}),\mathfrak{b}_1\rangle) =
    \frac{z}{\sqrt{2}}.
    $$This gives
    $$
    \varphi(z) \langle (I-\varphi(z)D)^{-1}(\mathfrak{c}),\mathfrak{b}_1\rangle = 0
    $$for every $z\in \bD$. Thus we get 
    $$
    \varphi(z) = \sqrt{2}\alpha_1 z.
    $$Since $\mathfrak{s}_2^\varphi\succeq 0$, $\sqrt{2}|\alpha_1|\leq 1$. We show now that for $c\in\overline{\bD}$, the only way $\mathfrak{s}_2^{cz}$ can be complete Pick is $c$ be unimodular. Indeed,
    \begin{align*}
        1-\frac{1}{\mathfrak{s}_2^\varphi} = 1 - \frac{(1-z\overline{w})^2}{1-|c|^2z\overline{w}} = (2-|c^2|-1)\frac{z\overline{w}(1-z\overline{w})}{1-|c|^2z\overline{w}} = (2-|c^2|-1) \sum_{n\geq1} \gamma_n(z\overline{w})^n
    \end{align*}where $\gamma_2 = |c|^2 -1<0$.

    \textbf{Case II:} If $\varphi(0)=a\neq 0$, then define $\widetilde\varphi(z) = m_a\circ\varphi$ where $m_a$ is the automorphism of $\bD$ given by $m_a(z)=(z-a)(1-\overline a z)^{-1}$. Then a plain computation shows that 
    $$
    \mathfrak{s}_2^{\widetilde\varphi}(z,w) = f(z)\overline{f(w)}\mathfrak{s}_2^{\varphi}(z,w)
    $$where $f$ is the non-vanishing function $f(z)=\sqrt{1-|a|^2}(1-\overline{a}\varphi(z))^{-1}$. For the details of the computation, see \cite[Proposition 2.4]{ADP_JGA2025}.  As discussed in Remark \ref{R:Normalization}, the relation above implies that $\mathfrak{s}_2^{\widetilde\varphi}$ is a complete Pick kernel if and only if so is $\mathfrak{s}_2^{\varphi}$. In Case I, we have proved that $\mathfrak{s}_2^{\widetilde\varphi}$ is complete Pick only if $\widetilde\varphi(z) = cz$ for a unimodular constant $c$. This will imply that $\varphi(z) = m_{-a}\circ(cz) = cm_{-a\overline{c}}(z)$.
\end{proof}
We end with a discussion about a family of concrete examples for the Bergman kernel to illustrate Theorem \ref{T:Main}.

\begin{example}
For every contractive $\varphi$ in $H^\infty$, we have the following factorization of the sub-Bergman kernel $\mathfrak{s}_2^\varphi$:
$$
\mathfrak{s}_2^\varphi(z,w)= 
\frac{1-\varphi(z)\overline{\varphi(w)}}{(1-z\overline{w})^2} = \left( (1-\varphi(z)\overline{\varphi(w)})\mathfrak{s}(z,w)\right)\cdot \mathfrak{s}(z,w) .
$$Let us consider the contractive analytic function given by
$$
\widetilde\varphi(z)= \overline{\varphi(\overline{z})}
$$for every $z\in\bD$, and the multiplication operator $M_{\widetilde\varphi}$ on $H^2$. With the de Branges-Rovnyak space $\cD_\varphi$-valued function
$$
v(z)=(I-M_{\widetilde
\varphi}M_{\widetilde
\varphi}^*)^{1/2}s_{\overline{z}}
$$we can write the de Branges-Rovnyak kernel $\mathfrak{s}^\varphi$ as
$$
\mathfrak{s}^\varphi(z,w) = \frac{1-\varphi(z)\overline{\varphi(w)}}{1-z\overline{w}} = \langle v(z),v(w)\rangle.
$$Thus we see that for every contractive $\varphi$ in $H^\infty$, the sub-Bergman kernel $\mathfrak{s}_2^\varphi$ admits a complete Pick factor - viz., the Szeg\"o kernel $\mathfrak{s}(z,w)$:
$$
\mathfrak{s}_2^\varphi(z,w) = \frac{\langle v(z),v(w)\rangle}{1-z\overline{w}}.
$$The Bergman kernel $\mathfrak{s}_2$ is of the form 
$$
1-\frac{1}{\mathfrak{s}_2(z,w)}=
g(z)\overline{g(w)} - f(z)\overline{f(w)}
$$with $g,f:\bD\to\bC$ given by
$g(z)=\sqrt{2}z$ and $f(z)=z^2$,
respectively. To apply Theorem \ref{T:Main}, let us assume $\varphi(0)=0$ to make $\mathfrak{s}_2^\varphi$ a normalized kernel at $0$. By Part (2) of Theorem \ref{T:Main}, there exists a contractive analytic function 
$$
\Phi:\bD\to\cB\left( \begin{bmatrix}
\bC \\
\cD_\varphi
\end{bmatrix},\begin{bmatrix}
\bC \\
\cD_\varphi\\
\cD_\varphi
\end{bmatrix} \right)
$$such that for every $z\in \bD$,
$$\Phi(\varphi(z))
\begin{bmatrix}
1\\
\sqrt{2}zv(z)
\end{bmatrix} = \begin{bmatrix}
\varphi(z) \\
v(z) \\
z^2v(z)
\end{bmatrix}
$$This function can be constructed as follows: Let us denote 
$$
\bold u(z)=\begin{bmatrix}
1\\
\sqrt{2}zv(z)
\end{bmatrix} \quad\mbox{and}\quad
\bold v(z) = \begin{bmatrix}
\varphi(z) \\
v(z) \\
z^2v(z)
\end{bmatrix}.
$$A computation yields the following:
$$
\frac{\langle \bold u(z),\bold u(w) \rangle - \langle \bold v(z),\bold v(w) \rangle}{1-\varphi(z)\overline{\varphi(w)}} =z\overline{w}.
$$By a rearrangement of the terms, we get
$$
\langle \bold u(z),\bold u(w) \rangle + z\varphi(z)\overline{w\varphi(w)} = z\overline{w}+\langle \bold v(z),\bold v(w) \rangle.
$$By the lurking isometry argument, this gives a contraction operator
$$
\begin{bmatrix}
A&B\\
C&D
\end{bmatrix}:\begin{bmatrix}
\sbm{\bC \\ \cD_\varphi} \\ \bC
\end{bmatrix}\to
\begin{bmatrix}
\sbm{\bC \\ \cD_\varphi \\ \cD_\varphi} \\ \bC
\end{bmatrix}
$$satisfying
$$
\begin{bmatrix}
\bold v(z) \\ z
\end{bmatrix}=\begin{bmatrix}
A& B\\
C& D
\end{bmatrix}\begin{bmatrix}
\bold u(z) \\ z\varphi(z)
\end{bmatrix} =\begin{bmatrix}
A\bold u(z) + B\varphi(z)z\\
C\bold u(z) + D\varphi(z) z
\end{bmatrix} .
$$By comparing the (21)-entries we get $z = (1-\varphi(z) D)^{-1}C\bold u(z)$. By putting this in the identity obtained by comparing the (11)-entries we get
$$
(A+\varphi(z)B(1-\varphi(z)D)^{-1}C)\bold u(z) = \bold v(z).
$$ Therefore if we set 
$$
\Phi(w) = A+wB(1-wD)^{-1}C.
$$Then $\Phi$ is the desired contractive analytic function mapping $\bold u(z)$ to $\bold v(z)$. Let us note that the argument given here is precisely the analysis required to prove Leech's Theorem -- see the original work \cite{Leech} and also \cite[Theorem 8.57]{AMc_PickBook}.

To see part (3) of Theorem \ref{T:Main}, we see by a plain computation that the kernel $\widetilde k$ as in \eqref{tildeK} in this case is given by
$$
\widetilde k(z,w) = \frac{1}{(1-\varphi(z)\overline{\varphi(w)})(1-z\overline{w})}.
$$Clearly $(1-\varphi(z)\overline{\varphi(w)})\widetilde k(z,w) = \mathfrak{s}(z,w)$ a complete Pick kernel.
\end{example}



\bibliographystyle{amsplain}

\end{document}